\documentclass{amsart}
\usepackage{amsfonts,amssymb,amsmath,amsthm}
\usepackage{url}
\usepackage{enumerate}
 \newtheorem{thm}{Theorem}[section]
 
 \newtheorem{lem}[thm]{Lemma}
 
 \theoremstyle{definition}
 
 \theoremstyle{remark}

 \numberwithin{equation}{section}
 \usepackage{times}

\newcommand{\R}{{\mathbb R}}
\newcommand{\Sn}{{\mathbb S}^{n-1}}
\newcommand{\Sph}{\mathbb{S}}
\newcommand{\Kn}{{\mathcal K}^n}

\newcommand{\N}{{\mathbb N}}

\newcommand{\Ha}{{\mathcal H}}
\newcommand{\D}{{\rm d}}

\begin{document}

%
%
%
%
%
%
%
%
%

\title[H\"older continuity for support measures]
 {H\"older continuity for support measures\\ of convex bodies}

\author[Daniel Hug]{Daniel Hug}

\address{%
Karlsruhe Institute of Technology\\ 
Department of Mathematics\\
D-76128 Karls\-ruhe\\ 
Germany}

\email{daniel.hug@kit.edu}

\thanks{The authors acknowledge support by the German research foundation (DFG) 
through the research group `Geometry and Physics of Spatial Random Systems' under  grant HU1874/2-1.}
\author{Rolf Schneider}
\address{Mathematisches Institut\\ 
Albert-Ludwigs-Universit{\"a}t\\
D-79104 Freiburg i. Br.\\ 
Germany}
\email{rolf.schneider@math.uni-freiburg.de}

\subjclass{Primary 52A20; Secondary 52A22}

\keywords{Support measure, curvature measure, area measure, weak convergence, H\"older continuity, stability}

\date{October 10, 2014}
\dedicatory{}

\begin{abstract}
The support measures of a convex body are a common generalization of the curvature measures and the area measures. With respect to the Hausdorff metric on the space of convex bodies, they are weakly continuous. We provide a quantitative improvement of this result, by establishing a H\"older estimate for the support measures in terms of the bounded  Lipschitz metric, which metrizes the weak convergence. Specializing the result to area measures yields a reverse counterpart to earlier stability estimates, concerning Minkowski's existence theorem for convex bodies with given area measure.\end{abstract}

\maketitle

\section{Introduction}\label{sec1}

In the theory of convex bodies in Euclidean space, the curvature functions (elementary symmetric functions of principal curvatures or radii of curvature), known from the differential geometry of hypersurfaces, have been replaced by curvature measures and area measures. Their common generalization, the support measures, take into account that a boundary point of a convex body and an outer unit normal vector at this point in general do not determine each other uniquely. The support measures of a convex body in Euclidean space $\R^n$ are Borel measures on the unit sphere bundle of $\R^n$, with the property that their marginal measures are the curvature measures on $\R^n$ and the area measures on the unit sphere $\Sn$. On the space of convex bodies with the Hausdorff metric, the support measures are weakly continuous. In the present note, we improve this statement by showing that the support measures are locally H\"older continuous with respect to the bounded Lipschitz metric.

We denote by $\Kn$ the space of convex bodies (nonempty compact convex subsets) in Euclidean space $\R^n$, as usual equipped with the Hausdorff metric $d_H$. We write $B^n$ for the unit ball in $\R^n$. For $\rho\ge0$, $K_\rho:=K+\rho B^n$ is the parallel body of the convex body $K$ at distance $\rho$. Let $\Lambda_i(K,\cdot)$ denote the $i$th support measure of $K\in\Kn$. Its definition, as well as the definition of the bounded Lipschitz metric $d_{bL}$, will be recalled in Section \ref{sec2}.

\begin{thm}\label{T1}
Let $K,L\in\Kn$ be convex bodies, and let $R$ be the radius of a ball containing $K_2$ and $L_2$. Then
\begin{equation}\label{eqT1} 
d_{bL}(\Lambda_i(K,\cdot), \Lambda_i(L,\cdot)) \le C(R)\, d_H(K,L)^{1/2}
\end{equation}
for $i\in\{0,\dots,n-1\}$, where $C(R)$ is a constant which (for given dimension) depends only on $R$.
\end{thm}

We shall obtain Theorem \ref{T1} by adapting an approach due to Chazal, Cohen--Steiner and M\'{e}rigot \cite{CCM10}. These authors have obtained similar estimates for local parallel volumes of compact sets and have deduced estimates for curvature measures of sets of positive reach. Our restriction to convex bodies, which are the natural sets for the consideration of support measures, allows a simpler approach which, for this restricted class of sets, yields a more general result. 

In Section \ref{sec4} we show that the H\"older exponent $1/2$ in the estimate (\ref{eqT1}) is best possible.  

A special case of Theorem \ref{T1} concerns the area measure $S_{n-1}(K,\cdot)$. If $\omega\subset\Sn$ is a Borel set, then $S_{n-1}(K,\omega)=2\Lambda_{n-1}(K,\R^n\times\omega)$. From Theorem \ref{T1} it follows under the same assumptions on $K$ and $L$ that
\begin{equation}\label{4.14} 
d_{bL}(S_{n-1}(K,\cdot),S_{n-1}(L,\cdot))\le C'(R)\,d_H(K,L)^{1/2}.
\end{equation}
We want to present some motivation for proving such an inequality.

The area measure is the subject of a famous existence and uniqueness theorem due to Minkowski (see, e.g., \cite[Sec.~8.2]{Sch14}). The uniqueness assertion has been improved by some stability results. One of these (going back to Diskant; see \cite{Diskant72}, \cite[Thm.~8.5.1]{Sch14}) says that for convex bodies $K,L\in \Kn$ one has
\begin{equation}\label{4.15} 
d_H(K,L') \le \gamma \,\|S_{n-1}(K,\cdot)-S_{n-1}(L,\cdot)\|_{\rm TV}^{1/n}
\end{equation}
for a suitable translate $L'$ of $L$, where $\|\cdot\|_{\rm TV}$ denotes the total variation norm. Here $\gamma>0$ is a constant depending only on the dimension and on  a-priori bounds for the inradius and circumradius of $K$ and $L$. 

The stability result (\ref{4.15}) has the flaw that the left side can be small even if the right side is large. For example, a unit cube $K$ and a rotated image $L$ of $K$ can have arbitrarily small Hausdorff distance and still satisfy $\|S_{n-1}(K,\cdot)-S_{n-1}(L,\cdot)\|_{\rm TV}\ge 1$. It seems, therefore, more meaningful to replace the right-hand side in (\ref{4.15}) by an expression involving a metric for measures that metrizes the weak convergence. For the L\'{e}vy--Prokhorov metric, such a stability result was proved in \cite{HS02}; see also \cite[Thm.~8.5.3]{Sch14}. It was deduced from a corresponding stability result for the bounded Lipschitz metric (which is implicit in the proof, though it was not stated explicitly), namely
\begin{equation}\label{4.16} 
d_H(K,L') \le \gamma\,d_{bL}(S_{n-1}(K,\cdot),S_{n-1}(L,\cdot))^{1/n}
\end{equation}
for a suitable translate $L'$ of $L$, with a constant $\gamma$ as above. It appears that the H\"older continuity (\ref{4.14}) is, in principle, a more elementary fact than its reverse, the stability estimate (\ref{4.16}), and should therefore have preceded it.

\section{Notation and preliminaries}\label{sec2}

We recall some notions and notation used in the following. The Hausdorff distance of two convex bodies $K,L\in\Kn$ is given by
$$ d_H(K,L) = \min\{\rho\ge 0: K\subset L_\rho,\,L\subset K_\rho\}.$$
Let $K\in\Kn$. The metric projection $p(K,\cdot):\R^n\to K$ is defined by letting $p(K,x)$, for $x\in\R^n$, be the unique point in $K$ for which $|p(K,x)-x|\le|y-x|$ for all $y\in K$, where $|\cdot|$ denotes the Euclidean norm. Further, $d(K,x)= |x-p(K,x)|$ is the distance of the point $x$ from $K$, and for $x\in \R^n\setminus K$, the vector $u(K,x)= (x-p(K,x))/d(K,x)$ is the unit vector pointing from $p(K,x)$ to $x$. In the following, we write $p(K,\cdot)= :p_K$, $u(K,\cdot)=:u_K$, and $d(K,\cdot)=:d_K$. For $\rho>0$, we set $K^\rho:=K_\rho\setminus K$, where $K_\rho=K+\rho B^n$ is the already defined parallel body of $K$ at distance $\rho$. The product space $\R^n\times \Sn$, with its standard Euclidean metric as a subspace of $\R^n\times\R^n$, is denoted by $\Sigma^n$. For $\eta\subset \Sigma^n$, we consider the local parallel set
$$ M_\rho(K,\eta):= \{x\in K^\rho: (p_K(x),u_K(x))\in\eta\}$$
and define 
\begin{equation}\label{2.1}
\mu_{K,\rho}(\eta):= \Ha^n(M_\rho(K,\eta)),
\end{equation}
where $\Ha^n$ denotes the $n$-dimensional Hausdorff measure. If $\eta$ is a Borel set, then $M_\rho(K,\eta)$ is a Borel set, and there is a polynomial expansion
\begin{equation}\label{4.9}  
\mu_{K,\rho}(\eta) =\sum_{i=0}^{n-1} \rho^{n-i}\kappa_{n-i}\Lambda_i(K,\eta)\qquad\mbox{for }\rho\ge 0,
\end{equation}
where the normalizing factor $\kappa_{j}$ is the $j$-dimensional volume of $B^{j}$; see \cite{Sch14}, formulas (4.4) and (4.18). This defines the {\em support measures} $\Lambda_0(K,\cdot), \dots, \Lambda_{n-1}(K,\cdot)$ of $K$. 
They are finite Borel measures on $\Sigma^n$. The measure $\Lambda_i(K,\cdot)$ is concentrated on ${\rm Nor}\,K$, the {\em normal bundle} of $K$. By definition, this is the subspace of $\Sigma^n$, with the induced topology, consisting of all pairs $(x,u)$ where $x$ is a boundary point of $K$ and $u$ is an outer unit normal vector of $K$ at $x$. 

The support measures have the property of weak continuity: if a sequence $(K_j)_{j\in\N}$ of convex bodies converges to a convex body $K$ in the Hausdorff metric, then the sequence $(\Lambda_i(K_j,\cdot))_{j\in\N}$ converges weakly to $\Lambda_i(K,\cdot)$. The topology of weak convergence can be metrized by the L\'{e}vy--Prokhorov metric $d_{LP}$ or by the bounded Lipschitz metric $d_{bL}$ (see, e.g., Dudley \cite[Sec.~11.3]{Dud02}). To define the latter, for bounded real functions $f$ on $\Sigma^n$ let
$$ \|f\|_L:=\sup_{a\not= b}\frac{|f(a)-f(b)|}{|a-b|},\qquad \|f\|_\infty:=\sup_a |f(a)|.$$
For finite Borel measures $\mu,\nu$ on $\Sigma^n$, their {\em bounded Lipschitz distance} is then defined by
$$ d_{bL}(\mu,\nu):= \sup\left\{ \left|\int_{\Sigma^n} f\,\D\mu- \int_{\Sigma^n} f\,\D\nu\right|:f\in \mathcal{F}_{bL}\right\},$$
where $\mathcal{F}_{bL}$ is the set of all functions $f:\Sigma^n\to\R$ with $ \|f\|_L\le 1$ and $\|f\|_\infty\le 1$.

\vspace{2mm}

\section{Proof of Theorem \ref{T1}}\label{sec3}

The following lemma is modeled after Proposition 4.1 of Chazal, Cohen--Steiner and M\'{e}rigot \cite{CCM10}. Under the restriction to convex bodies, it extends the latter to the measures $\mu_{K,\rho}$ defined by (\ref{2.1}).

\begin{lem}\label{L4.1}
If $K,L\in\Kn$ are convex bodies and $\rho>0$, then
\begin{align*}
d_{bL}(\mu_{K,\rho},\mu_{L,\rho}) &\le \int_{K^\rho\cap L^\rho} |p_K-p_L|\,\D\Ha^n+ \int_{K^\rho\cap L^\rho} |u_K-u_L|\,\D\Ha^n\\
&\qquad +\Ha^n(K^\rho\triangle L^\rho),
\end{align*}
where $\triangle$ denotes the symmetric difference.
\end{lem}

\begin{proof}
For $K\in\Kn$ and $\rho>0$, let $F_\rho:K^\rho\to\Sigma^n$ be defined by $F_\rho(x):= (p_K(x),u_K(x))$ for $x\in K^\rho$. Then  $F_\rho$ is continuous, and $\mu_{K,\rho}$ is the image measure of $\Ha^n$, restricted to the Borel subsets of $K^\rho$, under $F_\rho$. 

Let $f:\Sigma^n\to\R$ be a function with $\|f\|_L\le 1$ and $\|f\|_\infty\le 1$. Applying  the transformation formula for integrals to $F_\rho$ and  using the properties of $f$, we obtain, for $K,L\in\Kn$,
\begin{align*}
&  \left|\int_{\Sigma^n} f\,\D\mu_{K,\rho} -\int_{\Sigma^n} f\,\D\mu_{L,\rho} \right|\\
&\qquad  = \left|\int_{K^\rho} f\circ(p_K,u_K)\,\D\Ha^n -\int_{L^\rho} f\circ(p_L,u_L)\,\D\Ha^n\right|\\
& \qquad \le \int_{K^\rho\cap L^\rho} \left|f\circ(p_K,u_K)-f\circ(p_L,u_L)\right|\,\D\Ha^n\\
&  \qquad\qquad+\int_{K^\rho\setminus L^\rho}  \left|f\circ(p_K,u_K)\right|\,\D\Ha^n +\int_{L^\rho\setminus K^\rho}  \left|f\circ(p_L,u_L)\right|\,\D\Ha^n \\
& \qquad \le \int_{K^\rho\cap L^\rho} |(p_K,u_K)-(p_L,u_L)|\,\D\Ha^n + \int_{K^\rho\setminus L^\rho}  1\,\D\Ha^n +\int_{L^\rho\setminus K^\rho} 1\,\D\Ha^n \\
& \qquad \le \int_{K^\rho\cap L^\rho} (|p_K-p_L|+|u_K-u_L|)\,\D\Ha^n +\Ha^n(K^\rho\triangle L^\rho),
\end{align*}
from which the assertion follows.
\end{proof}

\noindent{\em Proof of Theorem} \ref{T1}.
Let $K,L\in\Kn$, and set $d_H(K,L)=:\delta$. Let $R$ be the radius of a ball containing $K_2$ ($=K+2B^n$) and $L_2$.
We assume that $\delta<1$; this is not a loss of generality, since the left side of (\ref{eqT1}) is bounded by a constant depending only on $R$. Let $0<\rho\le 1$. We use Lemma \ref{L4.1} and estimate the terms on the right-hand side. First, from Lemma 1.8.11 in \cite{Sch14} we get
\begin{equation}\label{4.10} 
\int_{K^\rho\cap L^\rho} |p_K-p_L|\,\D\Ha^n\le\sqrt{5D}\,\Ha^n(K^\rho\cap L^\rho)\sqrt{\delta} \le C_1(R)\sqrt{\delta},
\end{equation}
where $D={\rm diam}(K_\rho\cup L_\rho)$ and the constant $C_1(R)$ depends only on $R$.

Writing $g(x):= |u_K(x)-u_L(x)|$ if $x\in K^\rho\cap L^\rho$ and $g(x)=0$ otherwise, we have
\begin{equation}\label{4.51} 
\int_{K^\rho\cap L^\rho} |u_K-u_L|\,\D\Ha^n = \int_{K^\rho} g\,\D\Ha^n =\int_{K^\gamma} g\,\D\Ha^n + \int_{K^\rho\setminus K^\gamma} g\,\D\Ha^n
\end{equation}
with $\gamma:=\min\{\delta,\rho\}$. Clearly,
$$ \int_{K^\gamma} g\,\D\Ha^n \le C_2(R)\delta\le C_2(R)\sqrt{\delta}.$$
If $\delta\ge\rho$, the last integral in (\ref{4.51}) is zero. Let $\delta<\rho$. Using \cite{Sch14}, formula (4.38), we have
\begin{equation}\label{3.3} 
\int_{K^\rho\setminus K^\gamma} g\,\D\Ha^n = \sum_{j=0}^{n-1} \omega_{n-j} \int_\delta^\rho t^{n-j-1} \int_{\Sigma^n} g(x+tu)\,\Lambda_j(K,\D(x,u))\,\D t,
\end{equation}
where $\omega_k=k\kappa_k$. Let $(x,u)\in{\rm Nor}\,K$ and set $y:= x+tu$, with $\delta<t\le\rho$. Then
\begin{equation}\label{3.2} 
g(x+tu) \le 2\sin\alpha,
\end{equation}
where $2\alpha$ is the angle between $u_K(y)$ and $u_L(y)$. In fact, (\ref{3.2}) holds with equality if $y\in L^\rho$, and $g(y)=0$ if $y\notin L^\rho$.

Since $d_H(K,L)\le \delta$, the ball $B(x,\delta)=\{z\in\R^n:|z-x|\le \delta\}$ contains a point of $L$. Therefore $d(L,y)\le t+\delta$ and hence $p(L,y)\in B(y,t+\delta)$. Let $H^-$ be the supporting halfspace of $K$ with outer normal vector $u$. Then $L\subset H^-+\delta u$, hence $p(L,y)\subset B(y,t+\delta)\cap(H^-+\delta u)$. The largest possible angle between $u(K,y)$ and $u(L,y)$ is attained if $p(L,y)\in {\rm bd}\, B(y,t+\delta)\cap {\rm bd}\,(H^-+\delta u)$. This gives
$$ \sin\alpha\le\frac{\sqrt{\delta}}{\sqrt{t+\delta}}.$$
Since $t\le\rho\le 1$ in (\ref{3.3}), we conclude that
$$ \int_{K^\rho\setminus K^\gamma} g\,\D\Ha^n \le 2\sqrt{\delta}\sum_{j=0}^{n-1} \omega_{n-j} \int_\delta^\rho \frac{1}{\sqrt{t+\delta}}\,\D t\cdot \Lambda_j(K,\Sigma^n) \le C_3(R)\sqrt{\delta}.$$
Altogether we get
\begin{equation}\label{4.11} 
\int_{K^\rho\cap L^\rho} |u_K-u_L|\,\D\Ha^n\le C_4(R)\sqrt{\delta}.
\end{equation}

\vspace{2mm}

For the estimation of $\Ha^n(K^\rho\triangle L^\rho)$, let $x\in K^\rho\setminus L^\rho$; then $x\in K_\rho\setminus K$ and $x\notin L_\rho\setminus L$. If $x\in L$, then $d(K,x)\le\delta$, hence $x\in K_\delta\setminus K$. If $x\notin L$, then $x\notin L_\rho$ but $x\in K_\rho$, $K_\rho\subset(L_\delta)_\rho=L_{\rho+\delta}$, and hence $x\in L_{\rho+\delta}\setminus L_\rho$. It follows that
$$ K^\rho \setminus L^\rho \subset (K_\delta\setminus K) \cup (L_{\rho+\delta}\setminus L_\rho)$$
and hence 
\begin{align*}
\Ha^n(K^\rho\setminus L^\rho) &\le  \Ha^n(K_\delta)-\Ha^n(K) + \Ha^n(L_{\rho+\delta}) -\Ha^n(L_\rho)\\
&\le C_5(R)\delta\le C_5(R)\sqrt{\delta}.
\end{align*}
Here $K$ and $L$ can be interchanged, and together with (\ref{4.10}), (\ref{4.11}) and Lemma \ref{L4.1} this gives
\begin{equation}\label{4.12}
d_{bL}(\mu_{K,\rho},\mu_{L,\rho}) \le  C_6(R)\sqrt{\delta}.
\end{equation}

To deduce an estimate for the support measures, we apply the usual procedure (e.g., \cite{Sch14}, p. 213) and choose in (\ref{4.9}) for $\rho$ each of the $n$ fixed values $\rho_j=j/n$, $j=1,\dots,n$, and solve the resulting system of linear equations (which has a non-zero Vandermonde determinant), to obtain representations
$$ \Lambda_i(K,\cdot) = \sum_{j=1}^n a_{ij}\mu_{K,\rho_j},\qquad i=0,\dots,n-1,$$
with constants $a_{ij}$ depending only on $i,j$. Using the definition of the bounded Lipschitz metric, we deduce that
\begin{equation}\label{4.13} 
d_{bL}(\Lambda_i(K,\cdot),\Lambda_i(L,\cdot)) \le \sum_{j=1}^n|a_{ij}|d_{bL}(\mu_{K,\rho_j},\mu_{L,\rho_j}) \le C(R)\sqrt{\delta}.
\end{equation}
This completes the proof of Theorem \ref{T1}. \qed

\section{Optimality}\label{sec4}

The aim of this section is to show that the H\"older estimate of Theorem \ref{T1} is generally best possible, that is, the exponent $1/2$ in (\ref{eqT1}) cannot be replaced by a larger constant. Let $i\in \{1,\ldots,n-1\}$ and recall from \cite[(4.11), (4.18)]{Sch14} that 
$$
S_{i}(K,\cdot)= \frac{i\kappa_{n-i}}{\binom{n}{i}}\Lambda_{i}(K,\R^n\times\cdot)
$$ 
is the $i$-th area measure of $K$. For convenience, we use $\Psi_i(K,\cdot)=\Lambda_{i}(K,\R^n\times\cdot)$ in the following. Let $E$ be a fixed $(i+1)$-dimensional linear subspace of $\R^n$, let $B_E:=B^n\cap E$ be the unit ball and $\mathbb{S}_E:=\mathbb{S}^{n-1}\cap E$ the unit sphere in $E$. For $e\in \Sph_E$, $h\in (0,\pi/2)$ and $\tau\in\R$, 
let $H^-(e,\tau):=\{z\in\R^n: z\cdot e \le \tau\}$ 
and
$$
B_E(e,h):=B_E\cap H^-(e,\cos h).
$$

We assume first that $i\le n-2$. For $s\in [0,\pi/2]$, $t\in [0,h]$, $v\in \Sph_E\cap e^\perp$ and $w\in E^\perp\cap \Sn$, we define  $\varphi(s,t,v,w)\in\Sn$ by
$$
\varphi(s,t,v,w):=(\cos s\cos t) e+(\cos s\sin t) v+(\sin s) w. 
$$
We consider the function  $f:\Sn\to [0,\infty)$ given by
$$
f\left(\varphi(s,t,v,w)\right):=\cos s\left(1-\frac{\sin t}{\sin h}\right)
$$
if $(s,t,v,w)\in [0,\pi/2]\times [0,h]\times (\Sph_E\cap e^\perp) \times \Sph_{ E^\perp}$, and by $0$ otherwise. If $y=\varphi(s,t,v,w)$ and if $\pi_E(y)$ denotes the orthogonal projection of $y$ to $E$, then
$$
f(y)= \cos s -\frac{\cos s\sin t}{\sin h} =|\pi_E(y)|-\frac{|\pi_E(y)-(y\cdot e) e|}{\sin h}.
$$
This shows that the function $f$ is well-defined. 
Together with the fact that $f$ is zero on the boundary of its support, it also shows that there is a constant $c_1>0$ such that $\|f\|_L\le c_1/h$. Here and in the following, all constants can be chosen independently of $h$. Clearly, we have $\|f\|_\infty=1$. 

For $x\in\Sph_E$, we define $\nu(x)\subset\Sn$ by
$$
\nu(x):=\left\{(\cos s) x+(\sin s) w: w\in \Sph_{E^\perp},\, s\in [0,\pi/2]\right\}.
$$

By basic properties of area measures, we have
$$
\Psi_i(B_E,\cdot)=\frac{1}{\omega_{n-i}}\int_{\Sph^i}\int_{\nu(x)}\mathbf{1}\{u\in\cdot\}\, \mathcal{H}^{n-i-1}(\D u)\,\mathcal{H}^i(\D x)
$$
(a special case of \cite{GKW11}, Thm. 6.2) and
$$
\Psi_i(B_E(e,h),\omega)=\frac{\kappa_i}{\omega_{n-i}}\sin^i h\int_{\nu(e)}\mathbf{1}\{u\in\omega\}\, \mathcal{H}^{n-i-1}(\D u),
$$
if $\omega$ is a Borel set contained in the support of $f$. 

Integrating the function $f$ with these two measures, we get
\begin{eqnarray*}
& & \int_{\Sn} f(u)\,\Psi_i(B_E(e,h),\D u)\\
& & = \frac{\kappa_i \sin^i h}{\omega_{n-i}} \int_{\nu(e)} f(u)\,\Ha^{n-i-1}(\D u)\\
& & = \frac{\kappa_i \sin^i h}{\omega_{n-i}} \int_{\Sph_{E^\perp}} \int_0^{\pi/2} f((\cos s)e+(\sin s)w) \sin^{n-i-2}s\,\D s\, \Ha^{n-i-2}(\D w)\\
& & = \frac{\kappa_i \sin^i h}{\omega_{n-i}} \int_{\Sph_{E^\perp}} \int_0^{\pi/2} \cos s \sin^{n-i-2} s\,\D s\,\Ha^{n-i-2}(\D w) \\
& & =  \frac{\kappa_i \sin^i h}{\omega_{n-i}}\, \omega_{n-i-1}\int_0^{\pi/2}\cos s\sin^{n-i-2} s \,\D s\\
& & =  \frac{\kappa_i \kappa_{n-i-1}}{\omega_{n-i}}\sin^i h
\end{eqnarray*}
and
\begin{eqnarray*}
& & \int_{\Sn} f(u)\,\Psi_i(B_E,\D u)\\
& & = \frac{1}{\omega_{n-i}} \int_{\Sph_E} \int_{\nu(x)} f(u)\,\Ha^{n-i-1}(\D u)\,\Ha^i(\D x)\\
& & = \frac{1}{\omega_{n-i}} \int_{\Sph_E\cap e^\perp} \int_0^h\int_{\Sph_{E^\perp}} \int_0^{\pi/2} f(\varphi(s,t,v,w)) \sin^{i-1}t \sin^{n-i-2} s\\
& & \hspace{4mm} \times\enspace\D s\,\Ha^{n-i-2}(\D w)\,\D t \,\Ha^{i-1}(\D v)\\
& & = \frac{1}{\omega_{n-i}} \int_{\Sph_E\cap e^\perp} \int_0^h\int_{\Sph_{E^\perp}} \int_0^{\pi/2} \cos s\left (1-\frac{\sin t} {\sin h}\right)\sin^{i-1}t \sin^{n-i-2} s \\
& & \hspace{4mm} \times\enspace \D s\,\Ha^{n-i-2}(\D w)\,\D t \,\Ha^{i-1}(\D v)\allowdisplaybreaks\\
& & =\frac{\omega_i \omega_{n-i-1}}{\omega_{n-i}} \int_0^h \left(1-\frac{\sin t} {\sin h}\right)\sin^{i-1}t \,\D t
\int_0^{\pi/2}\cos s\sin^{n-i-2}s\,\D s\\
& & =\frac{\omega_i \kappa_{n-i-1}}{\omega_{n-i}}\left[\int_0^h  \sin^{i-1}t\,\D t -\frac{1}{\sin h}  \int_0^h  \sin^{i}t\,\D t \right].
\end{eqnarray*}
To estimate the last integrals, we observe that $0\le t-\sin t\le t^3/6$ for $t\in (0,\pi/2)$. Hence, by the mean value theorem,
$$
0\le t^k-\sin^k t\le \frac{k}{6}t^{k+2},\quad k\in\N_0, t\in (0,\pi/2).
$$
For $k\ge 1$ this yields 
$$
\frac{1}{k}h^{k}-\frac{k-1}{6(k+2)}h^{k+2}\le \int_0^h \sin^{k-1} t\, \D t\le \frac{1}{k}h^{k}.
$$
Moreover, since $0<h\le\sqrt{5}$, we also have 
$$
\frac{1}{h}\le\frac{1}{\sin h}\le \frac{1}{h-h^3/6}\le \frac{1}{h}(1+h^2).
$$
For $h$ tending to zero, we deduce that
$$ \int_{\Sn} f(u)\,\Psi_i(B_E,\D u)= \left[\frac{\kappa_i \kappa_{n-i-1}}{\omega_{n-i}}  - \frac{\omega_i \kappa_{n-i-1}}{(i+1)\omega_{n-i}}\right] h^i +O(h^{i+2})$$
and hence that
\begin{eqnarray*}
&& \int_{\Sn} f(u)\,\Psi_i(B_E(e,h),\D u)- \int_{\Sn} f(u)\,\Psi_i(B_E,\D u)\\
&&\qquad = \frac{\omega_i \kappa_{n-i-1}}{(i+1)\omega_{n-i}} h^i + O(h^{i+2}).
\end{eqnarray*}

Up to now, the vector $e\in \Sph_E$ was fixed, and the dependence of the function $\varphi$ on $e$ and of the function
$f$ on $e$ and $h$ was not emphasized. Now we write $\varphi=\varphi_e$ and $f=f_{e,h}$ and vary the vector $e$. We note that $h$ is the geodesic radius of the spherical cap  $B_E\cap H^+(e,\cos h)$, where $H^+(e,\tau):=\{z\in\R^n: z\cdot e \ge \tau\}$. Therefore, for given $h\in (0,\pi/2)$, we can choose $N=N_i(h)$ vectors $e_1, \dots,e_N \in \Sph_E$ with the property that the caps $B_E\cap H^+(e_j,\cos h)$, $j=1,\ldots,N$, are mutually disjoint and that $c_2\le N_i(h)h^i\le c_3$ with positive constants $c_2,c_3$, independent of $h$.

We define the function $f_h:\Sn\to[0,1]$ by $f_h=f_{e_j,h}$ on the image of $[0,\pi/2]\times [0,h] \times 
(\Sph_E \cap e_j)\times( E^\perp\cap \Sn)$ under  $\varphi_{e_j}$, for $j=1,\ldots,N$, and as zero otherwise. It is easy to see that $\|f_h\|_L\le c_4/h$ with a constant $c_4$. Further, let 
$$
B_E(h):=B_E\cap\bigcap_{j=1}^{N_i(h)}H^- (e_j,\cos h).
$$
Then we obtain 
\begin{align*}
&\int_{\Sn}f_h(u)\, \Psi_{i}(B_E(h),\D u)-\int_{\Sn}f_h(u)\, \Psi_{i}(B_E,\D u)\\
&\qquad =\sum_{j=1}^{N_i(h)}\left(\frac{\omega_{i}\kappa_{n-i-1}}{(i+1)\omega_{n-i}}h^{i}+O(h^{i+2})\right)\\
&\qquad =\frac{\omega_{i}\kappa_{n-i-1}}{(i+1)\omega_{n-i}}\,N_i(h)h^i+O(h^{2}).
\end{align*}
Since $\|f\|_L\le c_4/h$, we deduce that
$$
d_{bL}\left(\Psi_{i}(B_E(h),\cdot),\Psi_{i}(B_E,\cdot)\right)\ge c_5 h-c_6h^3,
$$
where $c_5,c_6$ are positive constants. Since clearly 
$$
d_H(B_E(h),B_E)\le 1-\cos h\le h^2,
$$
an estimate of the form
$$
d_{bL}\left(\Psi_{i}(B_E(h),\cdot),\Psi_{i}(B_E,\cdot)\right)\le c \,d_H(B_E(h),B_E)^\alpha,
$$
with some constant $c>0$ and arbitrarily small $h>0$, requires that $\alpha\le 1/2$.

So far, we have assumed that $i\le n-2$. The proof for $i=n-1$ follows the same lines, but is considerably simpler, since $E=\R^n$ in this case, and no dependence on the variable $s$ occurs (or, formally, we put $s=0$).


\end{document}